\documentclass[oneside,english,reqno]{amsart}

\usepackage{algorithm}
\usepackage{algorithmic}
\usepackage{multirow}
\usepackage{subfigure}
\usepackage{enumerate}

\usepackage[T1]{fontenc}
\usepackage[utf8]{inputenc}
\usepackage{babel}
\usepackage{amsmath}
\usepackage{amsthm}
\usepackage{amssymb}

\usepackage{comment}

\usepackage{graphicx}

\usepackage{graphicx}
\usepackage{tikz}
\usetikzlibrary{shapes.misc}
\usepackage{float}

\usepackage[unicode=true,pdfusetitle,
bookmarks=true,bookmarksnumbered=false,bookmarksopen=false,
breaklinks=false,pdfborder={0 0 1},backref=false,
hidelinks]
{hyperref}

\makeatletter
\numberwithin{equation}{section}
\numberwithin{figure}{section}

\makeatother

\newtheorem{theorem}{Theorem}[section]

\newtheorem{corollary}[theorem]{Corollary}

\newtheorem{definition}[theorem]{Definition}
\newtheorem{example}[theorem]{Example}

\newtheorem{lemma}{Lemma}[section]

\newtheorem{proposition}[theorem]{Proposition}
\newtheorem{remark}[theorem]{Remark}

\begin{document}

\title[Mangasarian-Fromovitz CQ and optimality conditions]{Mangasarian-Fromovitz-type constraint qualification  and optimality conditions for smooth infinite programming problems
}


\author{
	Ewa M. Bednarczuk$^1$
}
\author{
	Krzysztof W. Le\'sniewski$^2$
}
\author{
	Krzysztof E. Rutkowski$^3$ 
}
\thanks{$^1$ Warsaw University of Technology, 00-662 Warszawa, Koszykowa 75,\\ Systems Research Institute of PAS, PAS, 01-447 Warsaw, Newelska 6, 		 
}
\thanks{$^2$ Systems Research Institute of the Polish Academy of Sciences,	 
}
\thanks{$^3$ Cardinal Stefan Wyszy\'nski University, 01-815 Warsaw, Dewajtis 5, 		 \href{mailto:k.rutkowski@uksw.edu.pl}{k.rutkowski@uksw.edu.pl}.}

\begin{abstract}
We introduce a constraint qualification condition (GPMFCQ) for smooth  infinite programming problems, where the nonlinear operator defining the equality constraints has nonsurjective derivative at the local minimum. The condition is a generalization of PMFCQ introduced by Morduhovich and Nghia. We prove the existence of Lagrange multipliers by using either Hurwicz set or Nonlinear Farkas Minkowski condition.
\end{abstract}
\maketitle

\section{Introduction}

In this study, we introduce new qualification condition subject to the following smooth infinite programming problem 
\begin{align}
\begin{aligned}\label{prob:mainP}
&\min_{x \in E} f(x)\\
&\text{subject to } H(x) = 0\\
&\text{and } g_t(x) \leq 0,\quad t\in T,
\end{aligned}
\end{align}
where \( f: E \to \mathbb{R} \), \( H: E \to F \), and $ g_t: E \to \mathbb{R} $, $t\in T$ are continuously Fr\'echet differentiable mappings, with \( E, F \) being Banach spaces and $T$ is arbitrary.

The feasible set \( \mathcal{F} \) is defined as:
\begin{equation}\label{set:F}
  \mathcal{F} := \{ x \in E \ | \ H(x) = 0,\ g_t(x) \leq 0,\quad t\in T \}.  
\end{equation}

Such infinite programming problems with constraint sets defined by systems of  infinite (countable) number of inequalities and equalities with continuously differentiable functions defined on Banach spaces   were considered in  \cite{bednarczuk2023constraint}, where Schauder basis are used for representation of $\mathcal{F}$. Another approach is given by \cite{MR4159570}, where authors introduces new constraint qualifications based on continuity properties of the multi-valued mapping $\mathcal{M}(\cdot,\cdot)$ related to asymptotic KKT conditions. In \cite{MR3070104} Perturbed Mangasarian-Fromovitz Constraint Qualification (PMFCQ) together with Nonlinear Farkas Minkowski Constraint Qualification (NMFCQ) condition are introduced, which allow to prove the existence of Lagrange multipliers. An important requirement of PMFCQ is that the derivative $DH(\bar{x})$ of the operator $H$ at the local minimizer $\bar{x}$ is surjective.

The main goal of our paper to prove the existence of Lagrange multipliers in the situations where the derivative of mapping defining the equality constraints, \( DH(\bar{x}) \), does not need to be surjective. This generalization allows for a broader class of problems where the image of \( DH(\bar{x}) \) might not cover the entire codomain, making the analysis more intricate and the application more general. Additionally, the number of inequality constraints \( g_t(x) \leq 0 \), $t\in T$ can be arbitrary, including uncountable sets $T$. 

Our focus is on the generalization of the Mangasarian-Fromovitz Constraint Qualification (MFCQ) to scenarios where the derivative of the mapping that defines equality constraints is not necessarily surjective. This extension is particularly relevant for problems where the image of the derivative does not span the entire codomain, leading to a more nuanced analysis of feasible solutions. Additionally, we address scenarios with infinitely many inequality constraints, a situation commonly encountered in applications involving partial differential equations (PDEs) and other functional analysis contexts.

By introducing in Section \ref{section:gpfcq} Generalized Perturbed Mangasarian-Fromovitz Constraint Qualification (GPMFCQ),
 we provide a framework that accommodates these complexities. This new condition allows for the existence of Lagrange multipliers even when traditional assumptions are not met, thus offering a more flexible approach to deriving optimality conditions. GPMFCQ is a generalization of PMFCQ to the case when \( DH(\bar{x}) \) is not surjective at local minimizer $\bar{x}$ to problem \eqref{prob:mainP}. To the best of our knowledge, this is the first instance of such a constraint qualification being applied to both equality and inequality constraints in infinite-dimensional settings (for problems with only equality constraints see \cite{MR4104521}).

We base our results on \cite{MR4104521} for the mapping \( H \) and on \cite{bs2006variational} for the inequalities $g_t$, $t\in T$. 

In  \cite{MR4104521} author mainly uses the rank theorem (see \cite[Theorem 2.5.15]{manifolds_tensor_vol2}), while \cite{bs2006variational} introduces the PMFCQ and NFMCQ  conditions. Below, we briefly describe these conditions:

\begin{enumerate}
    \item[(EQ)] 

\textit{ Rank Theorem}: 
The rank theorem in infinite-dimensional spaces, as utilized  in \cite{MR4104521}, ensures that under certain conditions, the tangent space to the level set defined by \( H(x) = 0 \) can be described explicitly. The theorem assumes that the kernel of the derivative \( DH(x) \) is topologically complemented in \( X \), and the image of \( DH(x) \) is closed and topologically complemented in \( F \).
\item[(IQ)]
\textit{ PMFCQ and NFMCQ Conditions from \cite{bs2006variational}}:
\begin{enumerate}
\item \textit{ PMFCQ (Perturbed Mangasarian-Fromovitz Constraint Qualification)}; This condition assume that at point $\bar{x}\in E$ there exist a direction $\tilde{x}\in E$
from the kernel of derivative of equality constraint, 
for which all the constraints $\langle Dg_t(\bar{x}) , \tilde{x} \rangle$ for $t\in T_\varepsilon$ are  uniformly inactive.
\item \textit{ NFMCQ (Nonlinear Farkas-Minkowski  Constraint Qualification)}.
This condition assume that at point $\bar{x}\in E$ the cone generated by pairs $(\nabla g_t(\bar{x}), g_t(\bar{x}) )$ is weakly* closed for any $t\in T$.
\end{enumerate}
\end{enumerate}
Alternatively, (IQ) are investigated in Section \ref{section:Abadie} via the following conditions:
\begin{enumerate}
\item \textit{ GPMFCQ (Generalized Perturbed Mangasarian-Fromovitz Constraint Qualification, see Definition \ref{def:GPFCQ})};
\item \textit{ Weak*-closedness of Hurwicz set (see definition of set $\mathcal{M}(\bar{x},0)$ in Section 8 of \cite{bednarczuk2023constraint})}.
\end{enumerate}

The subsequent sections will delve into the theoretical underpinnings, leveraging advanced mathematical concepts like the Abadie constraint qualification and the infinite-dimensional rank theorem, to establish optimality conditions and solution existence for this extended problem framework.

\subsection{Our aims}

The primary aim of this study is to develop a comprehensive framework for addressing smooth infinite programming problems, those involving both equality and infinite number of inequality constraints. The specific objectives include:

\begin{enumerate}
    \item Extension of Classical Constraint Qualifications: To generalize the Perturbed Mangasarian-Fromovitz Constraint Qualification (PMFCQ) by introducing the Generalized Perturbed Mangasarian-Fromovitz Constraint Qualification (GPMFCQ). This new condition is designed to be applicable even when the derivative of the equality constraint mapping is not surjective, thus broadening the scope of problems that can be effectively analyzed.
\item 
Analysis of Infinite Inequality Constraints: To address the challenges posed by optimization problems that feature infinitely many inequality constraints. This aspect is particularly relevant for applications in fields such as control theory, optimal transport, and mathematical models governed by partial differential equations (PDEs), where the set of constraints can naturally become uncountably infinite.
\item 
Derivation of Optimality Conditions: To establish necessary optimality conditions for infinite programming problems under the newly proposed GPMFCQ. These conditions aim to ensure the existence of Lagrange multipliers, facilitating the analysis and solution of problems where standard constraint qualifications fail to provide meaningful results.
\end{enumerate}

\section{Preliminaries}

Let $\mathcal{F}:= \{ x \in E \mid H(x)=0,\ g_t(x)\leq 0, t\in T\}$, where $H:\ E\rightarrow F$, $g_t:\ E\rightarrow \mathbb{R}$, $t\in T$ are defined as in Introduction. The existence feasible solutions to \eqref{prob:mainP} requires a delicate balance between the regularity conditions and the structure of the infinite constraint set.

\begin{definition}
For a given, possibly nonconvex, set $Q\subset E$ and $\bar{x}\in Q$  the {\em tangent (Bouligand) cone} to $Q$ at  $\bar{x}\in Q$ is  defined as
\begin{equation} 
\label{eq_1} 
\mathcal{T}_{Q}(\bar{x}):=\{d\in E\ |\ \exists\ \{x_{k}\}\subset Q\,\ \{t_{k}\}\subset \mathbb{R}\,\ x_{k}\rightarrow x,\, t_{k}\downarrow 0,\ (x_{k}-\bar{x})/t_{k}\rightarrow d\},
\end{equation}
and the weak tangent cone to $Q$ at  $\bar{x}\in Q$ is defined as
\begin{equation*}
    \mathcal{T}_{Q}^w(\bar{x}):=\{d\in E\ |\ \exists\ \{x_{k}\}\subset Q\,\ \{t_{k}\}\subset \mathbb{R}\,\ x_{k}\rightarrow \bar{x},\, t_{k}\downarrow 0,\ (x_{k}-\bar{x})/t_{k}\rightharpoonup d\}.
\end{equation*}

Equivalently,  an element $d\in E$ belongs to the set $\mathcal{T}_{Q}(\bar{x})$ if there exists a function $r:\ [0,+\infty) \rightarrow E$, $\lim_{t\rightarrow 0^+} \frac{r(t)}{t}=0$ and $\varepsilon_0>0$ such that for all $t\in [0,\varepsilon_0)$ we have $\bar{x}+td+r(t)\in Q$.
\end{definition}

\begin{definition}
    For any $\bar{x}\in \mathcal{F}$, where $\mathcal{F}$ is given by   
    
    \begin{equation}
    \mathcal{F}:=\{x\in E\ \mid\ H(x)=0, \ \ g_{t}(x)\le 0\ \ t\in T\}.    
    \end{equation}
    $T(\bar{x})$ denotes the set of  active  (inequality) indices of $\mathcal{F}$ at $\bar{x}$, 
    \begin{equation*}
        T(\bar{x}):=\{ t\in T \mid g_t(\bar{x})=0 \}.
    \end{equation*}
 \end{definition}

\begin{definition}
The cone
\begin{equation} 
\label{eq_2}
\Gamma_\mathcal{F}(\bar{x}):=\{d\in E\ \mid DH(\bar{x})d=0,
 \langle Dg_{t}(\bar{x}),d\rangle\le 0,\ t\in T(\bar{x})\}
\end{equation}
is called the {\em linearized cone to $\mathcal{F}$ } at $\bar{x}$.
\end{definition}

In this paper we assume that at given point $\bar{x}\in E$,
\begin{equation}\label{assumption:aff}
    \operatorname{aff} \Gamma_{\mathcal{F}}(\bar{x}) = \{ d \in E \mid DH(\bar{x})d=0 \},
\end{equation} meaning that there is no $t\in T(\bar{x})$ such that 
$Dg_t(\bar{x})d=0$ for any $d\in \Gamma_{\mathcal{F}}(\bar{x})$. Let us note that it is not a restrictive assumption, it is made here to simplify the presentation.

\begin{definition}
Given a set $\Omega\subset X$ we define the Fr\'echet normal
cone and limiting/Mordukhovich normal cone to $\Omega$ at $\bar{x}\in \Omega$ by, respectively
\begin{align*}
    &\hat{N}_\Omega(\bar{x})= \{ v\in X^*\mid \forall x \in \Omega: \ \langle v,x-\bar{x}\rangle_X\leq o(\|x-\bar{x}\|_X),\\
    &N_\Omega(\bar{x})= \{ v\in X^*\mid \exists \{x_k\}\in \Omega\ \exists
    \{v_k\}\in X^*:\ x_k\rightarrow \bar{x},\ v_k\rightharpoonup v,\ v_k \in  \hat{N}_\Omega(x_k)\ \forall k\in \mathbb{N} \}.
\end{align*}
\end{definition}
\begin{corollary}(\cite[Corollary 1.11]{MR3070104})
Let $X$ be a reflexive Banach space, and let $\Omega\subset X$ with $x\in \Omega$. Then the Fr\'echet normal cone to $\Omega$ at $\bar{x}$, $\hat{N}_{\Omega} (\bar{x})$, is dual to the weak contingent cone to $\Omega$ at this point, $ \mathcal{T}_\Omega^w(\bar{x})$, i.e.,
\begin{equation*}
    \hat{N}_{\Omega} (\bar{x})= (\mathcal{T}_\Omega^w)^\circ(\bar{x}):= \{ x^* \in X^* \mid \langle x^*,z \rangle \leq 0, \ \text{whenever } v\in \mathcal{T}_\Omega^w(\bar{x}) \}.
\end{equation*}
Thus one has the duality relationship
\begin{equation*}
    \hat{N}_\Omega(\bar{x})=\mathcal{T}_\Omega^\circ(\bar{x})
\end{equation*}
when $X$ if finite dimensional.
\end{corollary}

\begin{theorem}(Lyusternik, see \cite[Section 0.2.4]{theory_of_external_problems_Ioffe} )\label{theorem:Ljusternik}
	Let $X$ and $Y$ be Banach spaces, let $U$ be a
	neighborhood of a point $x_0\in X$, and let $f:\ U \rightarrow Y$ be a Fr\'echet differentiable mapping. Assume that $f$ is regular at $x_0$, i.e., that
	$\text{Im}\, Df(x_0)= Y$,
	and that its derivative is continuous at this point (in the uniform operator
    	topology of the space $\Gamma(X, Y)$). Then the tangent space  $\mathcal{ T}_M(x_0)$ to the set
	\begin{equation*}
		M =
		\{x \in U \mid f(x) =
		f(x_0)\}
	\end{equation*}
	at the point $x_0$ coincides with the kernel of the operator $Df(x_0)$,
	\begin{equation}
		\label{lust1}
		\mathcal{T}_M(x_0) =
		\ker\, Df(x_0).
	\end{equation}
	Moreover, if the assumptions of the theorem are satisfied, then there exist a
	neighbourhood $U'\subset U$ of the point $x_0$, a number $K>0$, and a mapping
	$\xi \rightarrow x (\xi)$ of the set $U'$ into $X$ such that
	\begin{align}	\label{lust2}
		\begin{aligned}
			& f(\xi + x(\xi))= f(x_0),\\
			& \|x(\xi)\|\leq K \|f(\xi) - f(x_0)\|
		\end{aligned}
	\end{align}
	for all $\xi \in U'$.
	
\end{theorem}

\begin{theorem}\label{theorem:local_representation}(\cite[Theorem 2.5.14]{manifolds_tensor_vol2} Local Representation Theorem)
	Let $E,\ Y$ be 
	Banach spaces. 	Let $f:\ U \rightarrow Y$ be of class $C^r$, $r\geq1$ in a neighbourhood of $x_0\in U$, $U\subset E$ open set. Let $Y_1$ be a closed split image of $Df(x_0)$ with closed complement $Y_2$.
	Suppose that  $Df(x_0)$ has a split kernel $E_2=\ker Df(x_0)$ with closed complement $E_1$. 
	Then there are
	open sets $U_1\subset U \subset E_1\oplus E_2 $ and $U_2 \subset Y_1 \oplus E_2$, $x_0\in U_2$  and a $C^r$ diffeomorphism 
	$\psi:\ U_2\rightarrow U_1 $ such that $(f\circ \psi)(u,v)=(u,\eta(u,v))$ for any  $(u,v)\in U_1$, where 
	$u\in E_1$, $v\in E_2$ and  $\eta:\ U_2\rightarrow E_2$ is a $C^r$ map satisfying $D\eta(\psi^{-1}(x_0))=0$.
\end{theorem}

\begin{theorem}(Rank Theorem, see \cite[Theorem 2.5.15]{manifolds_tensor_vol2} )\label{theorem:rank}
	
	Let $E,\ Y$ be
	Banach spaces. 
	Let $x_0\in U$, where $U$ is an open subset of $E$ and $f:\ U\rightarrow Y$ be of class $C^1.$ 
	
	Assume that $Df(x_0)$ has a closed split image $Y_1$ with closed component $Y_2$ and a split kernel $E_2$ with closed component $E_1$ and that for all $x$ in some neighbourhood of $x_0$, $Df(x)|E_1:\ E_1 \rightarrow Df(x)(E)$ is an isomorphism.

	Then there exist open sets $U_1\subset Y_1\oplus E_2$, $U_2\subset E$, $V_1\subset Y$, $V_2\subset Y_1\oplus E_2$ and diffeomorphisms of class $C^1$, $\varphi:\ V_1\rightarrow V_2$ and $\psi:\ U_1\rightarrow U_2$,  $x_0=(x_{01},x_{02})\in U_2\subset U\subset E_1\oplus E_2$, i.e. $x_{01}\in E_1$, $x_{02}\in E_2$, $f(x_0)\in V_1$ satisfying 
	\begin{equation*}
		(\varphi \circ f \circ \psi )(w,e)=(w,0),\quad \text{where}\ w\in Y_1,\ e\in E_2
	\end{equation*} for all $(w,e)\in U_1$.
\end{theorem}

Let $\mathbb{R}^T=\{ (x_t)_{t\in T} \mid x_t\in \mathbb{R},\ t\in T \}  $, $\mathbb{R}_-^T=\{ (x_t)_{t\in T} \mid x_t\in \mathbb{R}_-,\ t\in T \}$ ,
$\tilde{\mathbb{R}}^T=\{ (x_t)_{t\in T} \in \mathbb{R}^T \mid x_t\neq 0 \text{ for finitely many } t\in T \}  $, $\tilde{\mathbb{R}}_-^T=\{ (x_t)_{t\in T} \in \mathbb{R}_-^T \mid x_t\neq 0 \text{ for finitely many} \ t\in T \}$, $K=\{0_F\}\times \mathbb{R}^T_-$. 

For completeness we include the description of $\tilde{\mathbb{R}}^T$ c.f. \cite[Proposition 6, page 32]{bourbaki2013general}).
\begin{lemma}\label{lemma:dual} Let $\mathbb{R}^T=\{ (x_t)_{t\in T} \mid x_t\in \mathbb{R},\ t\in T \}$ be the product space with the product topology (see \cite[$\mathsection$ 4.1]{bourbaki2013general}). The space $\tilde{\mathbb{R}}^T$ is isomorphic to the space of linear continuous functional defined on $\mathbb{R}^T$.

\end{lemma}

By \cite[Proposition 1, page 44]{bourbaki2013general}, since $g_t:\ E \rightarrow \mathbb{R}$ are continuous, there exists exactly one operator $G:\ E \rightarrow  F \times \mathbb{R}^T$ such that for any $x\in E$ and for any $t\in T$ we have $g_t(x)=p_t(G(x))$ and $p_F(G(x))=H(x)$, where $p_t:\ \mathbb{R}^T\rightarrow \mathbb{R}$ is projection onto the $t$-th component of $\mathbb{R}^T$ and $p_F:\ E \times F$ is projection onto $F$ component. 
Let $\bar{x}\in E$. Analogously, if $Dg_t(\bar{x}):\ E \rightarrow \mathbb{R}$ are continuous, then there exists exactly one operator $A:\ E \rightarrow F\times \mathbb{R}^T$ such that for any  for any $t\in T$ we have $Dg_t(\bar{x})=p_t(A(\bar{x}))$ and $DH(\bar{x})=p_F(A(\bar{x}))$. 
In this case we will denote $A(\bar{x})$ as $DG(\bar{x})$. 
If operator $G$ is differentiable at $\bar{x}$, then there exists $DG(\bar{x}):\ E \rightarrow F$
\begin{equation*}
    0=\lim_{h\rightarrow 0} \frac{G(\bar{x}+h)-G(\bar{x})-DG(\bar{x})h}{\|h\|}
\end{equation*}
and since
\begin{equation*}
    0=\lim_{h\rightarrow 0} \frac{(H(\bar{x}+h),(g_t(\bar{x}+h))_{t\in T})-(H(\bar{x}),(g_t(\bar{x}))_{t\in T})-A(\bar{x})h}{\|h\|},
\end{equation*}
$p_t(DG(\bar{x}))=Dg_t(\bar{x})$, $t\in T$, $p_F(DG(\bar{x}))=DH(\bar{x})$, 
by uniqueness of the operator $A(\bar{x})$ we have $A(\bar{x})=DG(\bar{x})$ and $DG(\bar{x})$ is continuous.

The following definition was introduced in \cite{MR4159570}.
\begin{definition}
A feasible point $\bar{x}\in \mathcal{F}$ of \eqref{prob:mainP} is called KKT point if there exists $\bar{\lambda}\in F^*\times \tilde{R}^T$ such that
\begin{equation*}
    0=Df(\bar{x})+DG(\bar{x})^*\bar{\lambda},\footnote{Here $DG(\bar{x})^*$ is an adjoint operator to $DG(\bar{x})$.}\quad \text{and}\quad \bar{\lambda} \in \hat{N}_K(G(\bar{x})).
\end{equation*}
\end{definition}

The following definition was introduced in \cite{MR3070104}.
\begin{definition}
(nonlinear Farkas-Minkowski constraint qualification) We say that system
\eqref{set:F}
satisfies the Nonlinear Farkas-Minkowski constraint
qualification (NFMCQ) at $\bar{x}$ if the set
\begin{equation*}
    \operatorname{cone} \{ (\nabla g_t(\bar{x}),\langle g_t(\bar{x}),\bar{x}\rangle-g_t(\bar{x}) \mid t\in T \}
\end{equation*}
is weak*-closed in product space $X^*\times \mathbb{R}$.
\end{definition}

\section{Generalized Perturbed Mangasarian-Fromovitz Constraint Qualification (GPMFCQ)}\label{section:gpfcq}

Let $\bar{x}\in \mathcal{F}$. We consider the situation when $DH(\bar{x})$ is not onto. In this section we introduce the Generalized Perturbed Mangasarian-Fromovitz Constraint Qualification (GPMFCQ), which provides a more flexible framework for analyzing the problem \eqref{prob:mainP}:

\begin{definition}(GPMFCQ) \label{def:GPFCQ}
Let $E,F$ Banach spaces and $\bar{x}\in E$. Let $E_2=\ker DH(\bar{x})$ and $F_1=\operatorname{im} DH(\bar{x})$. Let $H=(H_1,H_2):\ E_1 \oplus E_2\rightarrow F_1\oplus F_2$, where $H_1:\ E \rightarrow F_1=DH(\bar{x})E=DH(\bar{x})E_1$. We say that \textit{Generalized Perturbed Mangasarian Fromovitz Constraint Qualification} holds at $\bar{x}\in E$ if the following holds:
\begin{enumerate}
    \item[(EQ)]\label{gpmfcq:EQ} $E=E_1 \oplus E_2 $, where $E_2=\operatorname{ker} DH(\bar{x})$ is closed, $F=F_1\oplus F_2$, $F_1=\operatorname{im} DH(\bar{x})(E)$ is closed and there exists $U(\bar{x})$ such that
\begin{equation*}
    \operatorname{im} DH(x) \cap F_2 = \{ 0 \} \quad x \in U(\bar{x}),
\end{equation*}
    \item[(IQ)]\label{gpmfcq:IQ} there exists $\tilde{x}\in E$ such that $DH_1(\bar{x})(\tilde{x})=0$ and
\begin{equation*}
   \inf_{\varepsilon>0} \sup_{t\in T_\varepsilon(\bar{x})} \langle Dg_t(\bar{x}) , \tilde{x} \rangle <0 \quad \text{for all}\  t\in T_\varepsilon(\bar{x}):=\{ t \in T \mid g_t(\bar{x})\geq - \varepsilon \}.
\end{equation*}
\end{enumerate}
\end{definition}

In the case when $DH(\bar{x})$ is onto (i.e. (EQ) is automatically satisfied ) the above definition  reduces to the Perturbed Mangasarian Fromovitz Constraint Qualification (PMFCQ) introduced in \cite{MR3070104}.

By combining \cite[Theorem 1]{MR3070104} with the Rank Theorem \ref{theorem:rank} we obtain the following representations of normal cones under the GPMFCQ.

\begin{theorem}\label{theorem:cone_representation}

 Let $\bar{x}\in \mathcal{F}$,
Assume that GPMFCQ holds for $\mathcal{F}$ at $\bar{x}$.
Then
\begin{equation*}
    \hat{N}(\bar{x},\mathcal{F})=\bigcap_{\varepsilon>0} \operatorname{cl}^* \operatorname{cone}\{ Dg_t(\bar{x}),\ t\in T_\varepsilon(\bar{x})\} + (DH_1(\bar{x}))^*(F_1^*).
\end{equation*}
Moreover, if   $g_t:\ E \rightarrow \mathbb{R}\cup \{+\infty \}$, $t\in T$  are uniformly Fr\'echet differentiable at $\bar{x}$ then
\begin{equation*}
    N(\bar{x},\mathcal{F})=\hat{N}(\bar{x},\mathcal{F})=\bigcap_{\varepsilon>0} \operatorname{cl}^* \operatorname{cone}\{ Dg_t(\bar{x}),\ t\in T_\varepsilon(\bar{x})\} + (DH_1(\bar{x}))^*(F_1^*).
\end{equation*}
\end{theorem}
\begin{proof}
Let $\bar{x}\in \mathcal{F}$ and define $\mathcal{F}_1$ as
\begin{equation*}
    \mathcal{F}_1=\{ x \in E \mid H_1(x)=0,\quad g_t(x)\leq 0,\ t\in T \},
\end{equation*}
The proof reduces to showing that there exists $U(\bar{x})\subset E$ such that $\mathcal{F}\cap U(\bar{x})=\mathcal{F}_1\cap U(\bar{x})$.

By Rank Theorem \ref{theorem:rank} applied to $H$, there exists $U_1(\bar{x})\subset E$ and diffeomorphisms $\psi:\   F_1 \oplus E_2 \supset V(\bar{w},\bar{e}) \rightarrow U_1(\bar{x})\subset E_1 \oplus E_2 $, $\varphi:\  F_1 \oplus F_2 \supset V_1(H(\bar{x})) \rightarrow V_2 \subset F_1 \oplus F_2$ such that
\begin{equation}\label{eq:rank}
    \varphi \circ H \circ \psi (w,e) = (w,0) \quad \text{for all} \quad (w,e)\in V(\bar{w},\bar{e})=\psi^{-1}(U_1(\bar{x})), \ \bar{x}=\psi(\bar{w},\bar{e}). 
\end{equation}
Moreover,  by Local Representation Theorem \ref{theorem:local_representation}  there exists $\eta:\ F_1\oplus E_2 \supset V(\bar{w},\bar{e}) \rightarrow F_2$, $C^1$ function such that $H\circ \psi (w,e)=(w,\eta(w,e))$. By \eqref{eq:rank} we have
\begin{align}
    \begin{aligned}
\label{eq:rank2}
   H \circ \psi (w,e) &=  \varphi^{-1}(w,0)= (w,\eta(w,e))\\
   &\text{for all} \quad (w,e)\in V(\bar{w},\bar{e})=\psi^{-1}(U_1(\bar{x})), \ \bar{x}=\psi(\bar{w},\bar{e}). 
   \end{aligned}
\end{align}
Let $x\in U_1(\bar{x})\cap \mathcal{F}_1$. Then $H_1(x)=0$ and $x=\psi(w,e)$ for some $(w,e)\in V(\bar{w},\bar{e})$. By \eqref{eq:rank2},
\begin{equation*}
    0=H_1(x)=H_1\circ \psi (w,e)= w,
\end{equation*}
i.e., $w=0$. Moreover, by \eqref{eq:rank2}, 
\begin{equation*}
    H_2(x)=H_2(\psi (w,e))=\eta (0,e)=\varphi^{-1}(0,0).
\end{equation*}
Therefore $\eta(0,e)=\eta(0,\bar{e})$. Then
\begin{equation*}
    H_2(x)=H_2(\psi (w,e))=H_2(\psi (0,\bar{e}))=H_2(\bar{x})=0
\end{equation*}
Therefore $x\in \mathcal{F}$.

By \cite[Theorem 2]{MR3070104} applied to $\mathcal{F}_1$ we have that
\begin{equation*}
    \hat{N}(\bar{x},\mathcal{F}_1)= \bigcap_{\varepsilon>0} \operatorname{cl}^* \operatorname{cone}\{ Dg_t(\bar{x}),\ t\in T_\varepsilon(\bar{x})\} + (DH_1(\bar{x}))^*(F_1^*)
\end{equation*}
Since $\mathcal{F}\cap U_1(\bar{x})=\mathcal{F}_1\cap U_1(\bar{x})$ we obtain that
\begin{equation*}
    \hat{N}(\bar{x},\mathcal{F})=\hat{N}(\bar{x},\mathcal{F}_1)= \bigcap_{\varepsilon>0} \operatorname{cl}^* \operatorname{cone}\{ Dg_t(\bar{x}),\ t\in T_\varepsilon(\bar{x})\} + (DH_1(\bar{x}))^*(F_1^*).
\end{equation*}
\end{proof}
The following theorem of existence of Lagrange multipliers follows:
\begin{theorem}
Let $\bar{x}$ be local minimizer of the infinite program
\begin{equation}\label{prob:main}
    \min_{x\in \mathcal{F}} f(x), 
\end{equation}
where $f:\ E\rightarrow \mathbb{R}\cup \{ +\infty \}$ is Fr\'echet differentiable at $\bar{x}$. Assume that GPMFCQ holds at $\bar{x}\in \mathcal{F}$. Then
\begin{equation*}
    0\in \nabla f(\bar{x}) + \bigcap_{\varepsilon>0}\operatorname{cl}^* \operatorname{cone} \{ \nabla g_t(\bar{x}),\ t\in T_\varepsilon(\bar{x})\} + DH_1(\bar{x})^*(F_1^*).
\end{equation*}
Moreover, if NFMCQ holds at $\bar{x}$, then
\begin{equation*}
    0\in \nabla f(\bar{x}) +  \operatorname{cone} \{ \nabla g_t(\bar{x}),\ t\in T(\bar{x})\} + DH_1(\bar{x})^*(F_1^*).
\end{equation*}

\end{theorem} 
\begin{proof}
Since $\bar{x}$ is local minimizer of \eqref{prob:main} we have that $\bar{x}$ is also a local minimizer of the following problem
\begin{equation}\label{prob:main2}
        \min_{x\in \mathcal{F}\cap U_1(\bar{x})} f(x),  
\end{equation}
where $U_1(\bar{x})$ is such that rank theorem holds. By Fermat rule to problem \eqref{prob:main2} we obtain that
\begin{equation*}
    0 \in \hat{\partial}(f+\delta(\cdot;\mathcal{F}\cap U_1(\bar{x})))(\bar{x}),
\end{equation*}
where $\hat{\partial}$ denotes Fr\'echet differentiable. 
Since $f$ is Fr\'echet differentiable at $\bar{x}$, by Proposition 1.101 of [Mordukhovich 1] it follows that \begin{align*}
    0 \in \nabla f(\bar{x})+\hat{\partial}(\delta(\cdot;\mathcal{F}\cap U_1(\bar{x})))(\bar{x})=&\nabla f(\bar{x})+\hat{N}(\bar{x},\mathcal{F}_1\cap U_1(\bar{x}))\\
    =& \nabla f(\bar{x})+\hat{N}(\bar{x},\mathcal{F}).
\end{align*}
Then by Theorem  \ref{theorem:cone_representation} of representation of normal cone to $\mathcal{F}$ at $\bar{x}$ we have
\begin{equation*}
    0 \in \nabla f(\bar{x})+\bigcap_{\varepsilon>0} \operatorname{cl}^* \operatorname{cone}\{ Dg_t(\bar{x}),\ t\in T_\varepsilon(\bar{x})\} + (DH_1(\bar{x}))^*(F_1^*).
\end{equation*}
The second part of the assertion follows directly from definition of NFMCQ and the properties of of the sets $T_\varepsilon(\bar{x})$.
\end{proof}

In the case when $DH(\bar{x})$ is onto the above theorem reduces to \cite[Theorem 3]{MR3070104}.

\section{Abadie condition under GPMFCQ}\label{section:Abadie}

In this section we use GPMFCQ to prove Abadie condition at some $\bar{x}$, i.e. linearized cone to $\mathcal{F}$ at $\bar{x}$ is contained in the tangent cone to $\mathcal{F}$ at $\bar{x}$:
\begin{equation}\label{inclusion:Abadie}
\Gamma_\mathcal{F}(\bar{x})\subset T_\mathcal{F}(\bar{x})
\end{equation}

Let $\bar{x}\in\mathcal{F}$. Consider the function $h:E\rightarrow\mathbb{R}$,
\begin{equation} 
\label{eq:funkcjah} 
h(d):=\inf_{\varepsilon>0}\sup_{t\in T_{\varepsilon}}\langle D g_{t}(\bar{x}),d\rangle, 
\end{equation}
where $T_{\varepsilon}(\bar{x}):=\{t\in T\ |\ g_{t}(\bar{x})\ge-\varepsilon\}$.

We will prove under  continuity (with respect to $t$) of the gradients $Dg_{t}(\bar{x})$ that
\begin{equation} 
\label{eq:linearizedh} 
\Gamma_\mathcal{F}(\bar{x}):=\{d\in E\ \mid DH(\bar{x})d=0,
 \ h(d)\le 0\}.
 \end{equation}
 
 And moreover, if the function $h(\cdot)$ is u.s.c., then the relative interior\footnote{Recall that throughout the paper we assume \eqref{assumption:aff}.}
 $\operatorname{ri}$ of $\Gamma_\mathcal{F}(x)$
 \begin{equation} 
\label{eq:linearizedh1} 
\mbox{ri}\Gamma_\mathcal{F}(\bar{x}):=\{d\in E\ \mid DH(\bar{x})d=0,
 \ h(d)< 0\}.
 \end{equation}

 In the following proposition we provide the representation of the linearized cone to $\mathcal{F}$ at $\bar{x}$ in term of the function $h$ given by \eqref{eq:funkcjah}.
 \begin{proposition}
     Assume that $Dg_t(\bar{x})$ is continuous on $T$. Then $\Gamma_\mathcal{F}(\bar{x})=\{d\in E\ \mid DH(\bar{x})d=0,
 \ h(d)\le 0\}$.
 \end{proposition}
 
 \begin{proof}
 To show that $\Gamma_\mathcal{F}(\bar{x})\subset \{d\in E\ \mid DH(\bar{x})d=0,
 \ h(d)\le 0\}$, 
 by contradiction, take any $d\in E$ with $h(d)>0$, i.e. for some  $\kappa>0$
 \begin{equation*}
 \sup_{t\in T_{\varepsilon}}\langle Dg_{t}(\bar{x}),d\rangle>\kappa\ \ \text{for every } \varepsilon>0.
 \end{equation*}
 Since $T(\bar{x})\subset T_{\varepsilon_{2}}\subset T_{\varepsilon_{1}}$ for $\varepsilon_{2}\le \varepsilon_{1}$, there exists $\varepsilon_{0}>0$ such that
 \begin{equation*}
 \langle Dg_{t}(\bar{x}),d\rangle>\kappa\ \ \text{for every } t\in T_{\varepsilon}\subset T_{\varepsilon_{0}}.
 \end{equation*}
For each $\varepsilon_{n}\downarrow 0$ and each $\bar{t}\in T(\bar{x})$ there exists a sequence $t_{n}\in T_{\varepsilon_{n}}$ such that $t_{n}\rightarrow \bar{t}$ and
\begin{equation*}
\kappa<\langle Dg_{t_{n}}(\bar{x}),d\rangle\rightarrow \langle Dg_{\bar{t}}(\bar{x}),d\rangle\ge\kappa,
 \end{equation*}
 which shows that $d\not\in\Gamma_{\mathcal{F}}(\bar{x})$.
 
 On the other hand, let $d\in\{d\in E\ \mid DH(\bar{x})d=0,
 \ h(d)\le 0\}$, i.e. there exists $\varepsilon_{0}>0$ such that
\begin{equation*}
 \sup_{t\in T_{\varepsilon}}\langle Dg_{t}(\bar{x}),d\rangle\le 0. 
 \end{equation*}
 for all $\varepsilon<\varepsilon_{0}$. In particular,
 \begin{equation*}
 \langle Dg_{t}(\bar{x}),d\rangle\le 0 \ \ \ t\in T(\bar{x}),
 \end{equation*}
 i.e. $d\in \Gamma_\mathcal{F}(\bar{x}).$
 \end{proof}

 Now we will show that $h_\varepsilon(d):=\sup_{t\in T_{\varepsilon}} \langle D g_{t}(\bar{x}),d\rangle$ is u.s.c. for any $\varepsilon>0$.  Recall that a multifunction $\Gamma :\ X\rightrightarrows Y$ acting between two topological spaces $X,Y$, is u.s.c. at $x_0\in X$ if the following condition holds
 \begin{equation*}
     \forall_{\stackrel{V\supset \Gamma(x_0)}{V \text{ open}}} \exists_{U(x_0) \text{ open} }  \forall_{x\in U(x)}\quad  \Gamma(x)\subset V.
 \end{equation*}

 \begin{proposition} Assume that $Dg_t(\bar{x})$ is continuous $T$.
     If the multivalued mapping $\mathcal{T}:\ \mathbb{R}_{++} \rightrightarrows T$, where $\mathcal{T}(\varepsilon):= T_\varepsilon(\bar{x})=\{t\in T\ |\ g_{t}(\bar{x})\ge-\varepsilon\}$ is u.s.c. on $\mathbb{R}_{++}$, then the function the family of functions $h_\varepsilon:\ E\times \mathbb{R}$, $\varepsilon>0$ is u.s.c. for any $\varepsilon>0$.
 \end{proposition}
 
 \begin{proof}

 Let the function $\phi$ defined as $\phi(t,(d,\varepsilon)):=\langle D g_{t}(\bar{x}),d\rangle $. Then for any sequence $t_n\rightarrow t$ and for any sequence $d_n\rightarrow d$ we have
 \begin{align*}
     &|\langle D g_{t_n}(\bar{x}),d_n\rangle -\langle D g_{t}(\bar{x}),d\rangle | \\
     &= |\langle D g_{t_n}(\bar{x}),d_n\rangle - \langle Dg_t(\bar{x}) , d_n \rangle+\langle Dg_t(\bar{x}) , d_n \rangle - \langle D g_{t_n}(\bar{x}),d \rangle + \langle D g_{t_n}(\bar{x}),d \rangle -\langle D g_{t}(\bar{x}),d\rangle|\\
     &\leq \| Dg_{t_n}(\bar{x})-Dg_t(\bar{x})\|\|d_n\| + \| Dg_{t_n}(\bar{x})\|\|d_n-d\| + \|Dg_{t_n}(\bar{x})- Dg_{t}(\bar{x})\|\|d\|
     \rightarrow 0.
 \end{align*}
 Hence $\phi$ is continuous on $(t,(d,\varepsilon))\in T\times  E\times \mathbb{R}_{++}$, Define multifunction $\Gamma:\ E\times \mathbb{R}_{++} \rightrightarrows T$ as  $\Gamma((d,\varepsilon))=T_{\varepsilon}(x)$. By assumption the function $\Gamma$ is u.s.c. on $E\times \mathbb{R}_{++}$. The rest of the proof follows from [Theorem 2, page 116].
 \end{proof}

 \begin{corollary}
Assume that $Dg_t(\bar{x})$ is continuous on $T$ and that the multi-valued mapping $\mathcal{T}:\ \mathbb{R}_{++} \rightrightarrows T$, where $\mathcal{T}(\varepsilon):= T_\varepsilon(\bar{x})=\{t\in T\ |\ g_{t}(\bar{x})\ge-\varepsilon\}$ is u.s.c. on $\mathbb{R}_{++}$. Assume that  $\operatorname{aff} \Gamma(\bar{x}):= \{ d \in E   \mid DH(\bar{x})d=0\} $.
Then 
$\mbox{ri\,}\Gamma_\mathcal{F}(\bar{x}):=\{d\in E\ \mid DH(\bar{x})d=0,
 \ h(d)< 0\}$.
 \end{corollary} 
 \begin{proof}  

 Let us note that by assumption \eqref{assumption:aff}, $\operatorname{aff} \Gamma(\bar{x}):= \{ d \in E   \mid DH(\bar{x})d=0\} $. The proof follows from fact, that the infimum of family of u.s.c. functions is an u.s.c. function, therefore
 $h(\cdot)$ is u.s.c function on $E$ and
 $\mbox{ri\,}\Gamma_\mathcal{F}(\bar{x}):=\{d\in E\ \mid DH(\bar{x})d=0,
 \ h(d)< 0\}$, where relative interior of $ \Gamma(\bar{x})$ is an interior of $\{d\in E\ \mid 
 \ h(d)\leq  0\}$ on  $\operatorname{aff} \Gamma(\bar{x})$.
 \end{proof}
 
 The following Example provides sufficient conditions which guarantees the upper semi-continuity of the mapping $T_\varepsilon(\bar{x})$ as a multifunction of $\varepsilon$.

\begin{example}

 Let $\mathcal{T}(\varepsilon):=\{ t \in T=[a,b] \mid g_t(\bar{x})\geq - \varepsilon \}$ for any $\varepsilon>0$. Assume that 
\begin{enumerate}
    \item $g_t(\bar{x})$ is continuous as a function of $t\in T$,
    \item $g_t(\bar{x})$ has a uniformly continuous inverse as a function of $t$,
    \item exists $t\in T$ such that $g_t(\bar{x})=0$.
\end{enumerate}
Then for any $\varepsilon_n\rightarrow \varepsilon>0$ for any $\delta>0$ there exists $n_0$ such that for all $n\geq n_0$ we have

\begin{equation*}
    \mathcal{T}(\varepsilon_n) \subset \mathcal{T}(\varepsilon)+B(0,\delta).
\end{equation*}

Let us take any $\delta>0$. 
The assumption that there exists $t\in T$ such that $g_t(\bar{x})=0$ guarantees that $\mathcal{T}(\varepsilon)\neq \emptyset$ for any $\varepsilon>0$. 
First let us observe that if $\varepsilon_n < \varepsilon$ then the thesis follows by monotonicity of the mapping $\mathcal{T}$. Suppose that $\varepsilon_n\rightarrow \varepsilon$ is such that $\varepsilon_n>\varepsilon$. Let $t_n \in \mathcal{T}(\varepsilon_n)$ be a sequence. Without loss of generality we may assume that $g_{t_n}(\bar{x})\in [-\varepsilon_n,-\varepsilon)$ for all $n\in \mathbb{N}$.
Observe that, by assumptions, the set $\mathcal{T}(\varepsilon)$ is convex and closed.

Let $t_n^\prime = P_{\mathcal{T}(\varepsilon)}(t_n)$, $n\in \mathbb{N}$. Then $g_{t_n^\prime}(\bar{x})= - \varepsilon$. From the uniform continuity of inverse $g_t(\bar{x})$ we have:
    \begin{equation*}
        \forall_{\bar{\delta}>0} \ \exists_{\bar{\varepsilon}>0} \forall_{y_1=g_{s}(\bar{x}),\ y_2=g_{u}(\bar{x})} \quad |y_1-y_2|<\bar{\varepsilon} \implies |s-u| <\bar{\delta}.
    \end{equation*}
Since $\varepsilon_n\rightarrow\varepsilon$, by  uniform continuity of the inverse of $g_t(\bar{x})$ there exists $n_{\delta}\in \mathbb{N}$ such that for all $n \geq n_{\delta}$
\begin{align*}
    |t_n-t_n^\prime|<\delta \quad \text{if}\quad  |g_{t_n}(\bar{x})-g_{t_n^\prime}(\bar{x})|\leq |\varepsilon_n - \varepsilon|<\varepsilon_\delta.
\end{align*}
Therefore $ \mathcal{T}(\varepsilon_n) \subset \mathcal{T}(\varepsilon)+B(0,\delta)$ for all $n\geq n_\delta$. Note that by uniform continuity of the inverse $g_t(\bar{x})$, the choice $n_\delta$ does not depend upon the choice of sequence $t_n\in \mathcal{T}(\varepsilon_n)$.
\end{example}

The following theorem shows sufficient conditions for the Abadie CQ to  hold.

\begin{theorem}  \label{theorem:Abadie}
Let $(GPMFCQ)$ holds at $\bar{x}\in\mathcal{F}$ and assume that \eqref{assumption:aff} hold at $\bar{x}$. Let $\nabla g_{t}(\bar{x})$ be uniformly continuous with respect to $t\in T$. Then Abadie CQ \eqref{inclusion:Abadie} holds.
\end{theorem}

\begin{proof}   
By (GPMFCQ), $\mbox{ri\,}\Gamma_\mathcal{F}(\bar{x})\neq\emptyset$. Let $\bar{d}\in \mbox{ri\,}\Gamma_\mathcal{F}(\bar{x})$. By (GPMFCQ), there exists a neighbourhood of $\bar{x}$ such that
\begin{equation*}
    \operatorname{im} DH(x) \cap F_2 = \{ 0 \} \quad x \in U(\bar{x}),
\end{equation*}

By (EQ), \eqref{gpmfcq:EQ}, and  \cite[Proposition 3.4]{MR4104521}, the tangent cone to the set
\begin{equation*}
\mathcal{M}:=\{x\in E\ |\ H(x)=0\}\supset {\mathcal{F}}
\end{equation*}
is
\begin{equation*}
\mathcal{T}_{\mathcal{M}}(\bar{x})=\operatorname{ker}DH(\bar{x})=\mbox{aff\,}\Gamma_\mathcal{F}(\bar{x})\quad   (\mbox{in view of assumption \eqref{assumption:aff}}).
\end{equation*}
This means that $\bar{d}\in \Gamma_\mathcal{M}(\bar{x})$ because $\Gamma_\mathcal{F}(\bar{x})\subset \Gamma_\mathcal{M}(\bar{x})$, i.e. there exists $\frac{r(q)}{\|q\|}\rightarrow 0$ such that
\begin{equation*}
\bar{x}+q\bar{d}+r(q)\in \mathcal{M}\ \ \forall \ q\ \text{small enough}
\end{equation*}
Now, we show that
\begin{equation*}
g_{t}(\bar{x}+q\bar{d}+r(q))\le 0\ \ \forall\ q\ \text{small enough}.
\end{equation*}
Observe first that 
\begin{equation*}T_{\varepsilon_{1}}\subset T_{\varepsilon_{2}}\ \forall\  0\le\varepsilon_{1}\le\varepsilon_{2}
\end{equation*} 
Since $\bar{d} \in \operatorname{ri} \Gamma_\mathcal{F}(\bar{x})$ there exist $\varepsilon_0>0$ and $\kappa<0$ such that
\begin{equation*}
   \inf_{\varepsilon\geq \varepsilon_0} \sup_{t\in T_\varepsilon(\bar{x})} \langle Dg_t(\bar{x}) , \bar{d} \rangle < \kappa <0 \quad \text{where}\  T_\varepsilon(\bar{x}):=\{ t \in T \mid g_t(\bar{x})\geq - \varepsilon \}.
\end{equation*}
This means that there exists $\varepsilon_{1}\geq 
\varepsilon_0$ such that
\begin{equation*}
   \sup_{t\in T_\varepsilon(\bar{x})} \langle Dg_t(\bar{x}) , \bar{d} \rangle \le \sup_{t\in T_{\varepsilon_{1}}(\bar{x})} \langle Dg_t(\bar{x}) , \bar{d} \rangle<\kappa<0 \quad \text{for all}\ 
   \varepsilon_0\le\varepsilon\leq \varepsilon_{1}.
\end{equation*}
And consequently, for all $t\in T_{\varepsilon}$
\begin{equation*}
    \langle Dg_t(\bar{x}) , \bar{d} \rangle \le \sup_{t\in T_{\varepsilon_{1}}(\bar{x})} \langle Dg_t(\bar{x}) , \bar{d} \rangle<\kappa<0 \quad \text{for all}\ 
   \varepsilon_0\le\varepsilon\leq \varepsilon_{1}.
\end{equation*}
By the assumption of the uniform continuity of $Dg_{t}(\bar{x})$ with respect to $t\in T$,  there exists a neighbourhood $V$ of zero such that for each $t\in T_{\varepsilon}$, $\varepsilon_0\leq \varepsilon\leq \varepsilon_{1}$ and each $v\in V$
\begin{equation}\label{ineq:kappa}
    \langle Dg_{t}(\bar{x}+v), \bar{d}\rangle<\kappa<0.
\end{equation}

On the other hand, for each $t\in T_{\varepsilon_{1}}$ there exists $\theta_{t}\in(0,1)$ such that for all $0<q<\bar{q}$ small enough
\begin{align*}
g_{t}(\bar{x}+q\bar{d}+r(q)) &=
g_t(\bar{x})+\langle Dg_t(\bar{x}+\theta_t(q\bar{d}+r(q))), q\bar{d}+r(q) \rangle\\
&=g_t(\bar{x})+q [\langle  Dg_t(\bar{x}+\theta_t(q\bar{d}+r(q))), \bar{d}\rangle +  \langle  Dg_t(\bar{x}+\theta_t(q\bar{d}+r(q)), \frac{r(q)}{q} \rangle   ]\\
& < 0 
\end{align*}
since, by \eqref{ineq:kappa},  $Dg_t(\bar{x}+\theta_t(q\bar{d}+r(q)))\bar{d}<\kappa$ and  $Dg_t(\bar{x}+\theta_t(q\bar{d}+r(q))(\frac{r(q)}{q})< - \kappa$ for all $0<q\leq \bar{q}$ .
\end{proof}

Let us define Hurwicz set as follows
\begin{equation*}
	\mathcal{M}(\bar{x},0):= DG(\bar{x})^* \hat{N}_{K}(G(\bar{x})).
\end{equation*}
 
In a more general setting, this set has been defined by Hurwicz in  \cite{MR3204130} and a number of its properties has been shown in \cite{Kurcyusz1976}. We refer to this set as a \textit{Hurwicz set}. In the context of optimality condition Hurwicz set has been already used in \cite{bednarczuk2023constraint,MR4159570}. Since $K=\{0_F\}\times \mathbb{R}^T_-$ is a convex set we have that
\begin{equation*}
    \hat{N}_K(\bar{y})=\{ y^* \in (F\times \mathbb{R}^T)^* \mid \langle y^* , y-\bar{y} \rangle \leq 0\ \text{for all } y\in F\times \mathbb{R}^T\}.  
\end{equation*}
The following proposition provides sufficient conditions to show non-emptiness of Lagrange multipliers set.

\begin{proposition}\label{proposition:lagrange_multipliers}
	Let $E,F$ - Banach spaces. 
	Let $\bar{x}\in \mathcal{F}$ be a local minimizer of \eqref{prob:mainP}.
	Assume $(GPMFCQ)$ holds at $\bar{x}\in\mathcal{F}$ and \eqref{assumption:aff} hold at $\bar{x}$, and $D g_{t}(\bar{x})$ are uniformly continuous with respect to $t\in T$. 
	Assume that the Hurwicz set $\mathcal{M}(\bar{x},0)$ is weakly* closed. Then the set of Lagrange multipliers at $\bar{x}$ is nonempty.
\end{proposition}
\begin{proof}
By Theorem \ref{theorem:Abadie}, Abadie condition holds at $\bar{x}\in \mathcal{F}$. The rest of the proof follows the lines of the proof of Proposition 5.6 of \cite{MR4159570}. Indeed, $DG(\bar{x})^*\hat{N}_{K}(G(\bar{x}))\subset\{ y\in E^* \mid y=\sum_{i\in C} \lambda_i Dg_i^*(\bar{x}) + DH_1(\bar{x})^*w,\ C\subset T(\bar{x}) \text{ is finite}, \lambda_i\in \mathbb{R}_+,\ w \in F^* \}$.
\end{proof}

The following example illustrates Proposition \ref{proposition:lagrange_multipliers} under nonsurjectivity of $DH(\bar{x})$ and  noncompact, uncountable $T$.

\begin{example}
Let $E=\mathbb{R}^3$, $F=\mathbb{R}^3$. Define
\begin{align*}
    H(x_1,x_2,x_3)=\left[ \begin{array}{c}
         x_1^2  \\
         x_1+x_2 \\
         x_1^3+x_2
    \end{array}\right], \quad g_t(x_1,x_2,x_3)=tx_1^2+x_1-x_2+x_3,\ t\in (0,1)
  \end{align*}
and $\Omega=\{ x \in \mathbb{R}^3 \mid H(x)=0,\ g_t(x)\leq 0,\ t\in(0,1)\}$. Let $\bar{x}=(0,0,0)$. Then
\begin{align*}
        DH(x_1,x_2,x_3)=\left[ \begin{array}{ccc}
         2x_1 & 0 & 0  \\
         1 & 1 & 0 \\
         3x_1^2 & 1 & 0
    \end{array}\right], \quad  & Dg_t(x_1,x_2,x_3)=\left[ \begin{array}{c}
         2tx_1+1  \\
         -1 \\
         1
    \end{array}\right]\ t\in (0,1)\\
     DH(\bar{x})=\left[ \begin{array}{ccc}
         0 & 0 & 0  \\
         1 & 1 & 0 \\
         0 & 1 & 0
    \end{array}\right], \quad  &Dg_t(\bar{x})=\left[ \begin{array}{c}
         1  \\
         -1 \\
        1
    \end{array}\right]\ t\in (0,1)
\end{align*}
We will show that GPMFCQ holds for $\Omega$ at $\bar{x}$. We have
\begin{align*}
    &E_2:=\operatorname{ker} DH(\bar{x}) = \{ (0,0,s)\mid s\in \mathbb{R} \}, \quad F_1:= \operatorname{im} DH(\bar{x}) = \{ (0,s_1+s_2,s_2),\ s_1,s_2\in \mathbb{R} \},\\
    &E_1:=E_2^\perp=\{ (s_1,s_2,0) \mid s_1,s_2\in \mathbb{R} \},\quad F_2:=F_1^\perp= \{ (s,0,0)\mid s\in \mathbb{R} \}.
\end{align*}
Let $(x_1,x_2,x_3)$, $|x_1|<\frac{1}{\sqrt{3}}$, $x_2,x_3\in \mathbb{R}$   We have
\begin{equation*}
    \operatorname{im} DH(x_1,x_2,x_3) \cap F_2 = \{ (2x_1y_1, y_1+y_2, 3x_1^2y_1+y_2),\ y_1,y_2\in \mathbb{R} \} \cap \{ (s,0,0)\mid s\in \mathbb{R} \} = (0,0,0).
\end{equation*}
For any $x=(x_1,x_2,x_3)\in \mathbb{R}^3$,  $(0,x_1+x_2,x_2)=DH(\bar{x})(x)=DH_1(\bar{x})(x)$.
Note that $T_\varepsilon(\bar{x})=(0,1)$ for any $\varepsilon>0$. Let $\tilde{x}=(0,0,-1)$ then
\begin{equation*}
    DH(\bar{x})[-1,0,0]^T=\left[ \begin{array}{ccc}
         0 & 0 & 0  \\
         1 & 1 & 0 \\
         0 & 1 & 0
    \end{array}\right] \left[ \begin{array}{c}
         0  \\
         0 \\
         -1
    \end{array}\right] = \left[ \begin{array}{c}
         0  \\
         0 \\
         0
    \end{array}\right]
\end{equation*}
and 
\begin{equation*}
    \inf_{\varepsilon>0} \sup_{t \in T_\varepsilon(\bar{x})} \langle Dg_t(\bar{x}) ,   \tilde{x} \rangle =\inf_{\varepsilon>0} \sup_{t \in (0,1)} \langle (1,-1,1) , (0,0,-1) \rangle =-1 < 0.
\end{equation*}
Therefore GPMFCQ holds at $\bar{x}\in \Omega$. Therefore the normal cone to $\Omega$ at $\bar{x}$, by Theorem \ref{theorem:cone_representation} is
\begin{align*}
    \hat{N}(\bar{x},\Omega) &= \bigcap_{\varepsilon>0} \operatorname{cl}^* \operatorname{cone}\{ Dg_t(\bar{x}),\ t\in T_\varepsilon(\bar{x})\} + (DH_1(\bar{x}))^*(F_1^*)\\
    &= \operatorname{cone} (1,-1,1) + \{ (\lambda_1,\lambda_2,0), \lambda_1,\lambda_2 \in \mathbb{R} \} = \{ (x_1,x_2,x_3) \mid  x_1,x_2 \in \mathbb{R},\ x_3\geq 0 \}. 
    \end{align*}
    
    We have that the set
    \begin{equation*}
        \operatorname{cone} \{ (Dg_t(\bar{x}, \langle Dg_t(\bar{x}) , \bar{x} \rangle -g_t(\bar{x})  ),\ t\in T \} = \operatorname{cone} \{ (1,-1,1,0) \} = \{ (\lambda,-\lambda,\lambda,0),\ \lambda\geq 0\} 
    \end{equation*}
    is closed, hence NFMCQ holds at $\bar{x}$.
    
    Let us note that $\Omega=\{ (x_1,x_2,x_3) \mid x_1=x_2=0,\ x_3\leq 0 \}$. 
    Now let $f(x)=w(x_1,x_2)-u(x_3)$, where $w:\mathbb{R}^2\rightarrow \mathbb{R}$ is differentiable and $u:\ \mathbb{R}\rightarrow \mathbb{R}$ is differentiable such that $u^\prime(x_3)\geq 0$ for all $x_3\in \mathbb{R}$. Then $(0,0,0)$ is a local minimizer of $f$ on $\Omega$ and
    
\begin{equation*}
    0 \in \left[ \begin{array}{c}
         \frac{\partial w}{\partial x_1}  \\
         \frac{\partial w}{\partial x_2} \\
         -\frac{\partial u}{\partial x_3} 
    \end{array}\right] + \hat{N}(\bar{x},\Omega)= \left[ \begin{array}{c}
         \frac{\partial w}{\partial x_1}  \\
         \frac{\partial w}{\partial x_2} \\
         -\frac{\partial u}{\partial x_3} 
    \end{array}\right] + \{ (x_1,x_2,x_3) \mid  x_1,x_2 \in \mathbb{R},\ x_3\geq 0 \} .
\end{equation*}

\end{example}

The following proposition shows the representation of the normal cone $\hat{N}_{\mathcal{F}}(\bar{x})$ under assumption of GPFCQ and weak*-closedness of Hurwicz set.

\begin{proposition}
    	Let $E,F$ - Banach spaces. Assume $(GPMFCQ)$ holds at $\bar{x}\in\mathcal{F}$ and \eqref{assumption:aff} hold at $\bar{x}$, and $D g_{t}(\bar{x})$ are uniformly continuous with respect to $t\in T$. 
    	Then $\hat{N}_{\mathcal{F}}(\bar{x})=(\Gamma_{\mathcal{F}}(\bar{x}) )^\circ$. Moreover, if the Hurwicz set $\mathcal{M}(\bar{x},0)$ is weakly* closed, then 
    	$$\mathcal{M}(\bar{x},0)=\hat{N}_{\mathcal{F}}(\bar{x})=\bigcap_{\varepsilon>0} \operatorname{cl}^* \operatorname{cone}\{ Dg_t(\bar{x}),\ t\in T_\varepsilon(\bar{x})\} + (DH_1(\bar{x}))^*(F_1^*).$$ 
\end{proposition}
\begin{proof}
        By Theorem \ref{theorem:Abadie} we have that Abadie condition holds at $\bar{x}$, hence $(\Gamma_{\mathcal{F}}(\bar{x}) )^\circ= ((\mathcal{T}_{\mathcal{F}}^w(\bar{x}))^\circ=\hat{N}_{\mathcal{F}}(\bar{x})$. Moreover if Hurwicz set  $\mathcal{M}(\bar{x},0)$ is weakly* closed, then we have $DG(\bar{x})^*\hat{N}_{K}(G(\bar{x}))=\mathcal{M}(\bar{x},0)=(\Gamma_{\mathcal{F}}(\bar{x}) )^\circ=\hat{N}_{\mathcal{F}}(\bar{x})$ and by Theorem \ref{theorem:cone_representation},   $$\hat{N}_{\mathcal{F}}(\bar{x})=\bigcap_{\varepsilon>0} \operatorname{cl}^* \operatorname{cone}\{ Dg_t(\bar{x}),\ t\in T_\varepsilon(\bar{x})\} + (DH_1(\bar{x}))^*(F_1^*)=DG(\bar{x})^*\hat{N}_{K}(G(\bar{x})).$$

\end{proof}
\begin{remark}
Let us note that if NFMCQ hold at $\bar{x}$, then 
$$\bigcap_{\varepsilon>0} \operatorname{cl}^* \operatorname{cone}\{ Dg_t(\bar{x}),\ t\in T_\varepsilon(\bar{x})\} + (DH_1(\bar{x}))^*(F_1^*)= \operatorname{cone}\{ Dg_t(\bar{x}),\ t\in T(\bar{x})\} + (DH_1(\bar{x}))^*(F_1^*).$$
\end{remark}

Some relationships between the condition NFMCQ and weak*-closedness of Hurwicz set were investigated in \cite{bednarczuk2023constraint}.

\section{Examples}

In this section, we explore a selection of examples that aim to illustrate the theoretical concepts introduced earlier. These examples highlight how the Generalized Perturbed Mangasarian-Fromovitz Constraint Qualification (GPMFCQ) can be applied in different settings, from cases involving surjective derivatives $DH(\bar{x})$ to those with more complex, non-surjective conditions. Let $\bar{x} \in \mathcal{F}$.

\subsection{Surjectivity of derivative of equality constraints, inactivity of inequality}

Let $E$ - Banach space, $F=\mathbb{R}^n$. Suppose that $DH(\bar{x}):\ E\rightarrow F$ is onto, $E_2:=\ker DH(\bar{x})$ is closed, $E=E_1\oplus E_2$, $E_2$ - closed and $g_t(\bar{x})<0$, $t\in T$, where $T$ is finite.
In this case condition GPMFCQ is already satisfied, since (EQ) hold  with $F_1=F$ and (IQ) is satisfied\footnote{Note that the set $T_\varepsilon(\bar{x})$ is empty for $\varepsilon>0$ small enough.}.

Let
\begin{equation*}
    M:= \{ x \in E \mid H(x)=0\}
\end{equation*}
Then, $\Gamma_{\mathcal{F}}(\bar{x})=\Gamma_M(\bar{x})\subset \operatorname{ker} DH(\bar{x})$ and by Lyusternik, $ \operatorname{ker} DH(\bar{x})=T_{M}(\bar{x})$. Therefore for any $d \in \Gamma_{\mathcal{F}}(\bar{x})\subset T_{M}(\bar{x})$,
\begin{equation*}
    \exists \varepsilon_0>0\ \exists r(s):\ [0,\varepsilon_0) \rightarrow E\quad  \lim_{s\downarrow 0} \frac{\|r(s)\|}{s} =0 \quad H(\bar{x}+sd+r(s))=0.   
\end{equation*}
By assumption, $g_t(\bar{x})<0  $ for all $t\in T$ , we have 
\begin{align}\tag{P1}\label{assumption:P1}
    \forall r(s) \ \lim_{s\downarrow 0} \frac{\|r(s)\|}{s} =0\ \forall  d \in \Gamma_{\mathcal{F}}(\bar{x})\ \exists \varepsilon_1>0\  \forall s \in[0,\varepsilon_1) \quad g_t(\bar{x}+sd+r(s)) \leq 0 .
\end{align}
Then there exists $\varepsilon_2>0$ such that 
\begin{equation*}
 H(\bar{x}+sd+r(s))=0  \wedge \forall t\in T\ g_t(\bar{x}+sd+r(s))\leq 0, 
\end{equation*}
i.e. $d \in T_\mathcal{F}(\bar{x})$.

\subsection{Non-Surjectivity of the derivative of equality constraints, inactivity of inequality}
Assume that $F=F_1\bigoplus F_2$, where  $F_1=\operatorname{Im} DH(\bar{x})(E) $ - closed, $F_2$ - closed subspace of $F$, $E=E_1\oplus E_2$, $E_2$ - closed and $g_t(\bar{x})<0$, $t\in T$, where $T$ is finite. 
Suppose that
\begin{equation*}
    \operatorname{Im} DH(x) \cap F_2 = \{0\}
\end{equation*}
for $x\in U(\bar{x})$. Then GPMFCQ holds. 
Let
\begin{equation*}
    M:= \{ x \in E \mid H(x)=0\}.
\end{equation*}
Then, $\Gamma_{\mathcal{F}}(\bar{x})=\Gamma_M(\bar{x})\subset \operatorname{ker} DH(\bar{x})$ and by \cite[Proposition 3.4]{MR4104521}, $ \operatorname{ker} DH(\bar{x})=T_{M}(\bar{x})$.
Therefore for any $d \in \Gamma_{\mathcal{F}}(\bar{x})\subset T_{M}(\bar{x})$,
\begin{equation*}
    \exists \varepsilon_0>0\ \exists r(s):\ [0,\varepsilon_0) \rightarrow E\quad  \lim_{s\downarrow 0} \frac{\|r(s)\|}{t} =0 \quad H(\bar{x}+sd+r(s))=0.   
\end{equation*}
The condition \eqref{assumption:P1} hold. 
Then there exists $\varepsilon_2>0$ such that 
\begin{equation*}
 H(\bar{x}+sd+r(s))=0  \wedge \forall t\in T\ g_t(\bar{x}+sd+r(s))\leq 0, 
\end{equation*}
i.e. $d \in T_\mathcal{F}(\bar{x})$.

\subsection{Finite dimensional case with finite number of inequalities}
In this subsection we relate our investigations in case of $E$ finite dimensional case with finite number of inequalities $g_t$. The set $\mathcal{F}$ has the following form
\begin{equation*}
    \mathcal{F}= \{ x\in \mathbb{R}^n \mid H(x)=(H_1(x),\dots, H_k(x))=0,\ g_1(x)\leq 0,\dots , g_m(x)\leq 0 \}
\end{equation*}
where $H:\ \mathbb{R}^n \rightarrow \mathbb{R}^k$ is $C^1$ operator and $g_i:\ \mathbb{R}^n\rightarrow \mathbb{R}$, $i=1,\dots,m$ are $C^1$ functions. The condition (EQ) in this case is equivalent to the constant rank condition, i.e.
\begin{equation*}
    \exists U(\bar{x})\ \forall x \in U(\bar{x})\  \operatorname{rank} \{DH_1(x),\dots, DH_k(x) \} = \operatorname{rank} \{DH_1(\bar{x}),\dots, DH_k(\bar{x}) \}
\end{equation*}
If derivative $DH(\bar{x})$ is onto, then GPMFCQ is equivalent to MFCQ. Otherwise, GPFCQ condition is equivalent to the CRMFCQ, see \cite[Definition 5]{Kruger2014}. In this case Abadie condition holds (see \cite{Kruger2014}).

The following example illustrates the (EQ) condition in the above settings.

\begin{example}

Let \begin{equation*}H(x_1,x_2)=\left[ \begin{array}{c}
     x_1^2  \\
     x_2\\
     x_1^3+x_2
    \end{array}\right]\end{equation*}
be an operator $H:\ \mathbb{R}^2 \rightarrow \mathbb{R}^3$. Then
\begin{equation*}
    DH(x_1,x_2)=\left[ \begin{array}{cc}
     2x_1 & 0  \\
     0 & 1 \\
     3x_1^2 & 1
    \end{array}\right]
\end{equation*}
and
\begin{equation*}
    \operatorname{im} DH(x_1,x_2) = \{ (2x_1s_1,s_2,3x_1^2s_1+s_2),\ s_1,s_2 \in \mathbb{R}_+ \}
\end{equation*}
Let $\bar{x}=(1,0)$, we put $F_1=\operatorname{im} DH(1,0)=\{ (2s_1,s_2,3s_1+s_2),\ s_1,s_2\in \mathbb{R}\}=\operatorname{span} \{ (2,0,3), (0,1,1) \}$ and $y_2\in F_2 = \{ (\frac{3}{2}s,s,-s),\ s\in \mathbb{R}\}=\operatorname{span} \{ (\frac{3}{2},1,-1)\}$. There exists a neighbourhood  $U(\bar{x})$ such that 
\begin{equation*}
     \operatorname{im} DH(x_1,x_2) \cap F_2 = (0,0,0).
\end{equation*}
Since $\operatorname{rank} DH(1,0)$ is $2$ we can identify $F_1$ with $\mathbb{R}^2$ and $F_2$ with $\mathbb{R}$ and we write $y=(y_1,y_2)$, where $y_1\in F_1$, and $y_2\in F_2$. By rank theorem the respective $\varphi$ and $\psi$ exists.

In this case $E_2=(0,0)$, $E_1=\mathbb{R}^2$. 
Let us define $\psi(x,y)=(\sqrt{x},y)$, $x>0$, $y\in \mathbb{R}$. Then $H(\psi(x_1,x_2))=H(\sqrt{x_1},x_2)=(x_1,x_2,\sqrt{x_1}^3+x_2)$, $x_1>0$, $x_2\in \mathbb{R}$. Now define $\varphi(x,y,z)=(x,y,z-x^\frac{3}{2}-y)$ then $\varphi(H(\psi(x_1,x_2)))=(x_1,x_2,0)$. 

\end{example}

\section{Conclusions}

In this work, we have developed a framework for addressing optimization problems in infinite-dimensional spaces, focusing on scenarios that involve both equality and inequality constraints. In this approach we study the case, when the derivative of equality constraints is non necessarily surjective and the inequality constraints are indexed by set $T$, which is arbitrary. Important in this approach is the usage of the product topology in space $\mathbb{R}^T$ and its dual $\tilde{\mathbb{R}}^T$ (see Lemma \ref{lemma:dual}).

By introducing the Generalized Perturbed Mangasarian-Fromovitz Constraint Qualification (GPMFCQ), we extended classical constraint qualification conditions to accommodate cases where the derivative of the mapping defining equality constraints is not necessarily surjective. This generalization has proven to be particularly useful in situations where traditional assumptions do not hold, providing a new approach to deriving optimality conditions and ensuring the existence of Lagrange multipliers.

The theoretical results were supported by a series of examples, which illustrated the applicability of the GPMFCQ in various settings. These examples demonstrated the versatility of our approach and its potential to address challenges inherent in infinite-dimensional optimization problems, such as those encountered in control theory and partial differential equations.

\bibliographystyle{plain}
\bibliography{myarticlebibfile}
\end{document}